\documentclass[12 pt]{article}

\usepackage{amsmath, amsfonts, amssymb}

\usepackage[english]{babel}

\newtheorem{thm}{Theorem}

\newtheorem{prop}{Proposition}

\newtheorem{prob}{Problem}

\newtheorem{rem}{Remark}

\newtheorem{defin}{Definition}

\def\e {\varepsilon }

\def\sp {singular point }

\def\B {\mathcal B}

\def\zz{\mathbb Z}

\def \e {\varepsilon }

\def\D{\Delta }

\def\F {\mathcal F}

\def\homeo{homeomorphism }

\def\hol{holonomy }

\def\t {\tilde }

\def\st {such that }

\def\gas {genericity assumptions }

\newenvironment{proof}{{\noindent
\textbf{Proof}\,\,}}{\hspace*{\fill}$\Box$\medskip}

\def\be {\begin{equation}}

\def\ee {\end{equation}}

\def\qvfs {quadratic vector fields }

\def\qvf {quadratic vector field }

\def\nbd{neighborhood }

\def\cc{\mathbb C}

\def\fols{foliations }

\def\fol{foliation }

\def\pf {polynomial foliation }

\def\topeq {topologically equivalent }

\def\lvfs {linear vector fields }

\def\vfs {vector fields }

\def\vf {vector field }

\def\rig{rigidity }

\def\tes {there exists }

\def\A {\mathcal A}

\def\M{\mathcal M}

\def\l {\lambda }

\def\bbis {Baum-Bott indexes }

\def\mm{moduli map }

\def\G {\mathcal G}

\def\be {begin equation }

\def\ee {end equation } \def\g {\gamma }

\def\th {transversally holomorphic }

\title{Total rigidity of generic quadratic vector fields}

\author {Yu.Ilyashenko \thanks{ The first author was supported by
part by the grants NSF 0700973, RFBR 10-01-00739-а , RFFI-CNRS
050102801} \thanks {Cornell University, US; Moscow State and
Independent Universities, Steklov Math. Institute, Moscow} and
V.Moldavskis \thanks {Cornell University, US } }

\date{}

\begin{document}

\maketitle

\vskip 1pc

\emph{To the memory of Vladimir Arnold, a teacher, a leader, a poet
in mathematics}

\vskip 1pc

\begin{abstract}  We consider a class of foliations on the complex
projective plane that are determined by a \qvf in a fixed affine
neighborhood. Such foliations, as a rule, have an invariant line at
infinity. Two foliations with singularities on $\cc P^2$ are
topologically equivalent provided that there exists a homeomorphism
of the projective plane onto itself that preserves orientation both
on the leaves and in $\cc P^2$ and brings the leaves of the first
\fol to that of the second one. We prove that a generic \fol of this
class may be \topeq  to but a finite number of \fols of the same class, modulo
affine equivalence. This property is called \emph{total rigidity}.
Recent result of Lins Neto implies that the finite number above
does not exceed $240$.

This is the first  of the two closely related papers. It deals with
the rigidity properties of quadratic foliations, whilst the second
one studies the foliations of higher degree. \end{abstract}

\section{Introduction}   \label{sec:intro}

\subsection{Total rigidity}  \label{sub:tr} Total \rig is a property
opposite to structural stability. It is a formalization of the
following paradigm. \emph{For generic polynomial \fols in $\cc P^d$,
topological equivalence implies affine equivalence.} This is a
heuristic principle rather than a theorem. In the above form it is
proved only for  \lvfs of strict Siegel type whose normal form has a
non-trivial Jordan cell \cite{O96}.

A polynomial \vf in $\cc^d$ may be extended to a holomorphic line
field in $\cc P^d$; generically, the extended field has but   a
finite number of singular points. We will refer to it as a
polynomial \fol in $\cc P^d$, and denote by $\A_{n,d}$ the class of
all such foliations.

\begin{defin} \label{def:tr}   A \pf of class $\A_{n,d}$ is totally
rigid provided that \tes but a finite number of \fols of this class
(up to affine equivalence) that are \topeq to the foliation
considered. \end{defin} In what follows, $d = 2$, so we write $\A_n$
instead of $\A_{n,2}$.

\subsection{Main results}  \label{sub:main}

\begin{thm} \label{thm:tr}  A generic \fol of class $\A_2$, i.e.
quadratic \fol on $\cc P^2$, is totally rigid. \end{thm}

The first (computer assisted) proof of this theorem was obtained by
the second author  \cite{M06}. Here we present a purely analytic
proof of a slightly stronger result.
\begin{thm}  \label{thm:strong} Generic \qvf is \topeq to no more
than $240$ quadratic vector fields, modulo affine equivalence.
\end{thm}

An earlier estimate given in \cite{M06} is $14 \cdot 26^4$.
In subsection~\ref{sub:impr}, we state even a stronger result, Theorem~\ref{thm:sstrong}.
It is  proved together with Theorem~\ref{thm:strong}.
The sketch of the proof follows.

A generic \fol of class $\A_n$ has $n + 1$ singular points at
infinity, $O_1, \dots , O_{n+1}$, and $n^2$ singular points in
$\cc^2: Q_{n+2}, \dots , O_N \ N = n^2 + n + 1$.  For a singular
point $O$ of a \fol let $\l , \mu $ be the eigenvalues of the
linearization of the corresponding \vf at $O$. These eigenvalues are
defined up to a common factor. Let $\frac \l \mu + \frac \mu \l =
\nu (O)$. The number $\frac \l \mu $ is usually called the
characteristic number of $O$, and the number $\nu (O)$ is the
Baum-Bott index.

\begin{thm}   \label{thm:inv} The characteristic numbers
(equivalently, the Baum-Bott indexes) are topological invariants for
generic polynomial \fols of degree $n \ge 2$ in the projective
plane. \end{thm}

This theorem is proved in Section~\ref{sec:inv}.
The genericity assumptions are specified below, in subsection
\ref{sub:gen}.

Let $v$ be a polynomial \vf of degree $n$,  $$M(v) = (\nu (O_1),
\dots , \nu (O_N)) \in \cc^N$$ be the tuple of the \bbis for its
singular points; recall that $N = n^2 + n + 1$.

Affine equivalent \vfs as well as those that differ by a constant
factor, have the same \bbis at the corresponding singular points.
 Hence, the map $v \mapsto (\nu (O_1), \dots , \nu (O_N))$ may be
descended  to a map $$\mu : \t \A_n = \A_n/\mbox{Aff}\cc^2 \times \cc^* \to
\cc^N.$$

This map is called \emph{the moduli map}, see \cite{M06}, and its
image is denoted by $\M_n$. In \cite{LN}, the moduli map is called
the \emph{Baum-Bott} map. We preserve the term introduced in
\cite{M06}, in order to emphasize the relation of the map $\mu $ with
the moduli of topological classification of foliations.

\begin{thm}  \label{thm:mod} The moduli map is algebraic. For
generic \qvf $v$, the derivative of the moduli map at $v$ has the
full rank. \end{thm}

This theorem was proved in \cite{M06} by a computer assisted
calculation. At the same time, it follows from the result proved
analytically in \cite{G06}. In Section~\ref{sec:mod} we present a shorter analytic proof,
yet based on the idea suggested in \cite{G06}.
Lins Neto \cite{LN} considered a \mm of a larger space, namely, of a
space on \fols on  $\cc P^2$ that have $2$ tangent points with
generic lines, and not necessarily have the invariant line at infinity.

The class $\t \A_2$ is a subset of this space. Lins Neto proved that
the degree of the \mm on this larger space is exactly $240$. Hence,
the degree of the restriction of this map to $\t \A_2$ is no greater
than $240$.

Therefore, Theorems \ref{thm:inv} and \ref{thm:mod} imply
Theorem~\ref{thm:tr} and  Theorem~\ref{thm:strong}.

Indeed, the \mm is algebraic and of full rank at generic points.
Hence, the critical locus of the \mm is a proper algebraic
submanifold. It may contain an algebraic subset $\Sigma $ that is
blown down by the moduli map. This means that for any $\F \in \Sigma
, \mu^{-1}(\mu (\F (0))$ is an algebraic subset of positive
dimension, which is squeezed to a point by $\mu $. Any point $\F $
from $\t \A_2 \setminus \Sigma $ has no more than $240$ points in
the set $\mu^{-1}(\mu (\F (0))$. If, in addition, $\F $ satisfies the
genericity assumptions of Subsection~\ref{sub:gen},
then $\F $ is totally rigid and \topeq to no more than $240$
pairwise affine non-equivalent foliations.

\section{Topological invariance of Baum-Bott indexes
for generic foliations}  \label{sec:inv}

In this and the next section we prove an improved version of
Theorem~\ref{thm:inv}, Theorem~\ref{thm:sstrong} below,
and provide the genericity assumptions.

\subsection{Genericity assumptions} \label{sub:gen} Consider a
\fol $\F \in \A_n$ \st

 - it has exactly $n + 1$ singular points at infinity and $n^2$
 singular points in the fixed affine \nbd $\mathbb C^2$;

 - the monodromy group at infinity is non-solvable;

 - all the leaves are dense.

Let us check than these are genericity assumptions indeed.

The genericity of assumption on the singular points follows from the
Bezout Theorem. The fact that generic polynomial \vfs of degree $n
\ge 2$ have non-solvable monodromy at infinity was proved by several
authors \cite{S84}, \cite{LNSS99}, see also \cite{S06}. For \qvfs a
much stronger result is known.

\begin{thm} [Pyartli, \cite{P06}] \label{thm:piar} For any $\l =
(\l_1, \l_2, \l_3), \ \l_1 + \l_2 + \l_3 = 1$ denote by $\B_\l $ the
set of all \qvfs with the characteristic numbers $\l_1, \l_2, \l_3$
of the singular points at infinity. Let all the $\l_j$ be non-real.
Then any set $\B_\l $ contains no more than seven classes of affine
equivalence whose points correspond to \fols with solvable monodromy
group at infinity.   \end{thm}

For future use, denote the representatives of these $7$ classes
of affine equivalence by $A_j(\l ), \ j = 1, \dots , 7$.

Pyartli theorem stated in \cite{P06}, Theorem 2, is stronger than
Theorem~\ref{thm:piar}. Namely, the numbers $\l_j$ may be real, but
the following restrictions hold. If $\mbox{Re }\l_1 \ge \mbox{Re
}\l_2 \ge \mbox{Re }\l_3$, then $\l_1, \l_2 \notin \frac 1 4 \mathbb
Z \cup \frac 1 6\mathbb Z, \ \l_3 \ne \frac 1 3, \ \l_3 \ne \frac 1
4$. But this improvement of Pyartli theorem does not improve our
result, because density property requires that the numbers $\l_j$
are non-real.

Density of leaves requires non-solvable monodromy at infinity plus
hypebolicity of singular points at infinity \cite{S84}, \cite{N94},
\cite{S06}. The genericity of the first property was already
discussed. The second one determines real Zariski open set.

\subsection{Improvement of the main result}  \label{sub:impr}

We can now include genericity assumptions stated
above in Theorem ~\ref{thm:strong}. This gives us the following result.

\begin{thm} \label{thm:sstrong} For any tuple $\l = \l_1, \l_2, \l_3
, \ \l_1 + \l_2 + \l_3 = 1, \ \mbox{Im } \l_j \ne 0$, in any set
$\B_\l$ of \qvfs with the tuple of eigenvalues of singular points at
infinity equal to $\l $, \tes a Zarisski open subset $T_\l \subset
\B_\l $  \st any \fol $\F \in T_\l $ is \topeq to no more than $240$
\fols of class $\A_2$, modulo affine equivalence. The difference
$\B_\l \setminus T_\l $ equals to $ (\cup_1^7A_j(\l ))\cup \Sigma_\l
$, where $\Sigma_\l $ is the set of all points in $\B_\l $ that
belong to a subset, which is blown down by the moduli map. \end{thm}

This theorem follows from Theorem~\ref{thm:mod} proved in  Section \ref{sec:mod},
 and Theorem~\ref{thm:inv1}, improved version of Theorem~\ref{thm:inv}.
 The latter theorem is proved in this section.

As noticed in \cite{LN}, the moduli map blows down the family of
quadratic Darboux foliations with the first integral of the form
$(xy + x + y){(x - ky)}^\alpha = c$. In this family $\alpha $ is
fixed, and $k \in \cc $ is a parameter.

\subsection{Invariance Theorem}  \label{sub:inv1}

\begin{thm}  \label{thm:inv1} Suppose that \fol $\F \in \A_n$
satisfies \gas of subsection~\ref{sub:gen}. Then its \bbis are
topological invariants in the following sense. Let $\G \in \A_n$ be
\topeq to $\F $. The conjugacy induces a bijection $h: \mbox{sing }
\F \to \mbox{sing } \G $. Then for all $j = 1, \dots , N, \ \nu
(O_j,\F ) = \nu (h(O_j),\G )$. \end{thm}

\begin{proof}
{\bf Step $1$. Conjugacy of monodromy maps.} Let $\F $ and $\G $ be
conjugated by a \homeo $H$. Then $H$ topologically conjugates their
monodromy groups at infinity, denoted by $G_\F $ and $G_\G $, see
\cite{IYa07} Proposition 28.2. In more details, for any set $f_1,
\dots , f_n$ of generators of $G_\F $ \tes a set of generators $g_1,
\dots , g_n$ of $G_\G $ and a germ of a \homeo $h: (\cc ,0) \to (\cc
,0)$ \st

\begin{equation} \label{eqn:conj} f_j \circ h = h \circ g_j, \ j =
1, \dots , n. \end{equation}

\begin{thm}[\cite{S84}, \cite{N94}] \label{thm:conj} If two finitely
generated non-solvable groups of germs $(\cc ,0) \to (\cc ,0)$ are
topologically conjugated by an orientating preserving homeomorphism,
then this \homeo is in fact holomorphic. \end{thm}

{\bf Step $2$. Induced maps of cross-sections and transversal
holomorphy.} Arguments of step $2$ and $3$ are very close to those
of \cite{IYa07}, Lemma $28,24$. Yet we can not literally refer to
the lemma, so we present the arguments here.

Suppose that two \fols $\F $ and $\G $ are topologically equivalent.
Then for any point $p$ and any two germs of cross-sections: $\Gamma
$ at $p$, $\Gamma' $ at $p' = H(p)$, $\Gamma $ and $\Gamma' $ being
transversal to the leaves  of the \fols $\F $ and $\G $
respectively, \tes an induced germ of the \homeo \begin{equation}
\label{eqn:indu} h: (\Gamma ,p) \mapsto (\Gamma' ,p') \end{equation}
defined in the following way. Consider two flow boxes: of the
foliation $\F $ near $p$ and of the foliation $\G $ near $p'$. The
local leaves in these flow boxes are in one to one correspondence
with subdomains of $\Gamma $ and $\Gamma' $ respectively. The \homeo
$H$ sends the leaves of the first flow box to those of another one.
This induces the \homeo \eqref{eqn:indu}.

Note that germ $h$ that conjugates the monodromy groups in
\eqref{eqn:conj} is induced in a sense of the previous paragraph.

The induced germs respect the holonomy. namely, let $\g $ be a
nontrivial loop with the endpoint $p$ on a leaf of $\F $, and
$\Delta_{\g , \F }: (\Gamma ,p) \to (\Gamma ,p)$ be the germ of its
\hol transformation. Let $\g'  = H\g $, and $\Delta_{\g' ,\G }:
(\Gamma' ,p') \to (\Gamma' ,p')$ be the corresponding holonomy. Let
$h$ be the induced germ from \eqref{eqn:indu}. Then \begin{equation}
\label{eqn:conj1} h \circ \Delta_{\g ,\F } = \Delta_{\g' ,\G } \circ
h. \end{equation}

\begin{defin}  A \homeo $H$ that conjugates two \fols is called
\emph{transversally holomorphic}, if all the induced germs
\eqref{eqn:indu} are holomorphic. It is called \th at $p$, if the
germ $h$ in \eqref{eqn:indu} is holomorphic. \end{defin}

{\bf Step $3$. Extending transversal holomorphy.}
Theorem~\ref{thm:conj} implies that the \homeo $H$ is \th near
non-singular points of the infinity leaf. Density of leaves allows
us to extend the transversal holomorphy to all the non-singular
points of $\F $. We will use the following obvious remark. \emph{If
a \homeo is \th at all points of one cross-section of a flow box,
then it is \th at all points of this flow box.}

Now let us extend the transversal holomorphy of $H$ to all the
non-singular points of $\F $. Take any such point $q$ and a curve
$\g $ on the leaf passing through $q$ that connects $q$ with some
point $p \in \Gamma $, where $\Gamma $ is a cross-section to the
infinite leaf of $\F $, on which the relation \eqref{eqn:conj}
holds. Recall that $h$ in \eqref{eqn:conj} is a conformal map
induced by $H$. We may assume that $\g $ is non-self intersecting.
In the opposite case we delete from $\g $ all the loops produced by
self intersections. Then \tes a \nbd of $\g $ on the leaf of $\F $
which is biholomorphic equivalent to a disk. This \nbd may be
included in a flow box, for which one of the cross-sections belongs
to $\Gamma $. This allows us to conclude that $H$ is \th at $q$.
Hence, $H$ is \th everywhere.

{\bf Step $4$. Topological invariance of Baum-Bott indexes.}  For
any nondegenerate singular point $O$ of a complex planar foliation,
\tes a holomorphic separatrix $S$ through $O$. For the case when the
characteristic number is non-positive, this follows from the complex
version of the Hadamard-Perron theorem, \cite{IYa07}, theorem $7.1$.
For the case when this number is positive, this follows from the
Poincar\' e-Dulac theorem, ibid., Theorem $5.5$. Consider  a
nontrivial small loop $\g_O$ on $S$ around $O$ and a corresponding
holonomy transformation $\Delta_O$. It is well known that $\D'_O(0)
= \exp 2\pi i\frac \l \mu $ where $\frac \l \mu $ is a
characteristic number of $O, \ \mu $ corresponds to the tangent
vector to $S$ at $O$. Let $O' = H(O), \ \frac {\l' }{\mu' }$ be the
corresponding characteristic number, and $\D_{O'}$ be a holonomy map
of $\G $ corresponding to the loop $\g' = H(\g )$. The \homeo $H$ is
transversally holomorphic. Hence, the maps $\D_O$ and $\D_{O'}$ are
complex conjugated. This implies the coincidence of their
derivatives at zero. Hence, the characteristic numbers of $O$ and
$O'$ coincide modulo $\zz $. This is almost the desired statement.

\begin{prop} [\cite{I78}, \cite{Na82}]  \label{prop:forced} Suppose that two
planar \fols in a \nbd of a non-degenerate singular point are
topologically equivalent, and the corresponding holonomy maps are
analytically conjugate. Then the characteristic numbers of these
singular points coincide. \end{prop}

Together with the previous arguments, this proposition implies the
theorem. \end{proof}

\begin{rem} Note that all hyperbolic singular points of planar \fols
are topologically equivalent.

In particular, two singular points whose characteristic numbers
differ by a nonzero integer are topologically equivalent.
Proposition~\ref{prop:forced} claims that this equivalence can not
be transversally holomorphic.    \end{rem}

\section{Moduli map for quadratic vector fields}  \label{sec:mod}

In this section we will check that the dimensions of the factorized
space of \qvfs and the moduli space are both equal $5$. We will
prove that there exists at least one point where the \mm has the
full rank. This will imply Theorem~\ref{thm:mod}.

\subsection {Counting dimensions} \label{sub:count} The space of all
quadratic polynomial in the plane has dimension $6$. Hence, the
space of all planar \qvfs equals $12$. The affine group action on the phase
space induced a transformation on the space of quadratic vector fields.
Moreover, multiplication of a \vf by a nonzero number preserves the foliation.
The affine group has dimension $6$. Hence, $${\mbox{dim}}_\cc \A_2/\mbox{Aff}\ \cc^2
\times \cc^* = 5.$$ This is the ``effective'' dimension of the
space of quadratic vector fields. Denote by $\t \A_2$ the
factor-space above: \begin{equation} \label{eqn:factor} \t \A_2 =
\A_2/\mbox{Aff}\ \cc^2 \times \cc^*. \end{equation} Any class of
\vfs from $\t \A_2$ with at least three singular points in $\cc^2$
has a \emph{regular representative} with singular points $(0,0),
(2,0), (0,2)$.

The moduli space for \qvfs belongs to $\cc^7$. Yet it is subject to
at least two relations. The Baum-Bott equality \cite{BB72} for \qvfs
in $\cc P^2$ implies: \begin{equation}   \label{eqn:bb} \sum_1^7 \nu
(O_j) = 2.  \end{equation}

On the other hand, for the singular points at infinity
\begin{equation} \label{eqn:cs} \sum_1^3 \frac {\l_j}{\mu_j}(O_j) =
1. \end{equation} Here $\mu_j$ are eigenvalues that correspond to eigenvectors
tangent to the line at infinity.

This equality is usually called Camacho-Sad \cite{CS82}.

So, the dimension of the image of the \mm is no more than $5$. We
will prove that at a special point the \mm has rank $5$ indeed.
As mentioned before, this will imply Theorem~\ref{thm:mod}.

\subsection{Counting the rank}  \label{sub:rk} In this subsection we
will prove Theorem~\ref{thm:mod}.

\begin{proof} The idea of the proof goes back to \cite{G06}. In
fact, Theorem~\ref{thm:mod} was proved in \cite{G06}, but not
explicitly stated. Our proof is shorter and more explicit.

Suppose that, contrary to the statement of Theorem \ref{thm:mod},
rank of the moduli maps drops everywhere. Then the fibers of this
map are analytic sets of dimension at least one. Take such a fiber
passing through a special vector field $v_0$ chosen below. By the
curve selection lemma, \tes an analytic curve $\g = \{ v_\e | \e \in
(\cc ,0)\} $ \st $M(v_\e ) \equiv M(v_0), \ v_\e  \not \equiv v_0$.
Note that the \bbis for all the fields $v_\e $ are the same.

Take a \qvf $v_0$ having three invariant lines. It is affine
equivalent to a field with invariant lines $ x = 0, y = 0, x + y =
1$. Without loss of generality, we may assume that all the vector
fields $v_\e$ have singular points $(0,0), (0,1), (1,0)$. This may
be achieved by an affine transformation.

\begin{prop} \label{prop:inv} The assumption $M(v_\e ) \equiv M(v)$
implies that the fields  $v_\e$  have the same invariant lines as
$v_0$ for all $\e$. \end{prop}

\begin{proof} The invariance of the line $l$ passing through two
singular points of a \qvf $v$ is established like follows. Take any
nonsingular point $p \in l$. The line $l$ is invariant for $v$ iff
$v(p) \in T_p l$. Hence, the set of \qvfs with the singular points
$(0,0), (1,0), (0,1)$  for which the line $l = \{ y = 0 \}$ is
invariant has codimension one. Denote this set by IL, for
\emph{invariant line}.

The line $y = 0$ contains a third singular point of the extended
\fol $v_0$ at infinity. On the other hand, let $U \subset \t \A_2$
be a small \nbd of $v_0$, $O_4 = (0,0), \ O_5 = (1,0)$, and $ O_1$
be an infinite singular point of the \vf $v \in U$, close to the
infinite \sp of $v_0$ that belongs to $l$. Then the set of \vfs $v
\in U$ \st

\begin{equation} \label{eqn:cs1} \frac {\lambda_4}{\mu_4}(O_4,v) +
\frac {\lambda_5}{\mu_5}(O_5,v) + \frac {\lambda_1}{\mu_1}(O_1, v) = 1 \end{equation}
has codimension $1$.  Here $\mu_j$ are eigenvalues that correspond
to eigenvectors tangent to $l$. Denote this set by CS, for Camacho and Sad.
Now, invariance of the line $l$
for $v$ implies \eqref{eqn:cs1}. Hence, CS $\supset $ IL. But both
algebraic sets have codimension one. It is easy to see that both
sets CS and IL are irreducible near $v_0$. Hence, CS $\cap U =
\mbox{IL} \cap U$. By assumption, $\g \subset CS$. Hence, $\g
\subset $ IL. Therefore, the line $l$ is invariant for all the
fields $v_\e $. The same argument shows that the other two lines $x
= 0, \ x + y = 1$ are invariant for $v_\e $ for all $\e $.
\end{proof}

It is well known that a quadratic vector field having three
invariant lines has a multivalued Darboux first integral. For $v_\e$
it is

$$ F (x, y, \e) = x^{a(\e)} y^{b(\e)} (x + y - 1)^{c(\e)}. $$

Dividing all the exponents by a non-vanishing one preserves the first
integral. Hence, we may assume that $c(\e) = 1$. Then $a(\e)$ and
$b(\e)$ are the characteristic numbers at the points $(1,0)$  and
$(0,1)$. By assumption, they do not depend on $\e$. Hence, the
Darboux integral above does not depend on $\e$, and the foliations
determined by the \vfs $v_\e$ coincide, a contradiction. \end{proof}

This implies Theorem~\ref{thm:mod}. Together with
Theorem~\ref{thm:inv1}, and the quoted result of Lins Neto
, see Subsection~\ref{sub:main}, this implies our main result,
Theorem~\ref{thm:strong}.

\section{Open problem: total rigidity for higher degrees}
\label{sec:open}

Theorem~\ref{thm:inv} holds for polynomial \fols of arbitrary
degree. Yet the target of the \mm $\M_n$ has the same dimension as
$\t \A_n = \A_n/\mbox{Aff} \cc^2 \times \cc^*$ for $n = 2$, and
smaller dimension for larger $n$. Indeed, $$\mbox{dim} \t \A_n =
(n+1)(n+2)-7,$$ $$\mbox{dim} \M_n \le n^2 + n - 1.$$  The difference
is \begin{equation}  \label{eqn:dim} \mbox{dim} \t \A_n - \mbox{dim}
\M_n = 2n - 4. \end{equation}  It is $0$ for $n = 2$ only. For
larger $n$ the arguments above fail.

Yet the moduli space may be extended. Indeed, topological
equivalence of generic \fols implies analytic equivalence of the
corresponding monodromy groups. We used a very weak corollary of
this equivalence, namely, the coincidence of the linear terms of the
monodromy only. Equivalence relations on higher jets imply relations
on the higher Taylor coefficients of the monodromy maps. It may be
shown that the equivalence relations on the quadratic and cubic
terms generate extra $2n - 3$ moduli, thus compensating the gap
\eqref{eqn:dim}. This generates extended moduli map $\t \mu $.

\begin{prob} \label{prob:rank} Is it correct that the extended \mm
has full rank at a generic point? \end{prob}

\section{Acknowledgements}  \label{sec:ackn} The first author is grateful to the
organizers of the ``School in holomorphic \fol and dynamical
systems'', Mexico, August 2010, especially to Laura Ortiz and
Ernesto Rosales; to Adolfo Guillot, who introduced him in the recent
progress in the study of the Baum-Bott map; to all the participants
of the above mentioned School for the enthusiastic and creative
atmosphere, and to the UNAM that provided an ideal environment for
writing of this paper. Both authors are grateful to the Cornell
University, where the first theorem on total rigidity of quadratic
\fols was proved by the second author in the thesis \cite{M06} done
under the advisory of the first one.


\begin{thebibliography}{88}

\bibitem[1]{BB72} P. Baum, R. Bott, Singularities of holomorphic
    foliations,
    J.Differential Geometry, 7 (1972), 279-342


\bibitem[2]{CS82} C. Camacho, P. Sad, Invariatnt varieties through
    singularities of holomorphic vector fields, Ann. of Math., 115
    (1982), 579--595


\bibitem[3]{G06} A. Guillot, Semicompleteness of homogeneous
    quadatic vector fields, Ann. Inst. Fourier (Grenoble), 56
    (2006), 1583--1615


\bibitem[4]{I78} Yu. Ilyashenko, The topology of phase portraits of
    analytic
    differential equations in the
      complex projective plane, Trudy sem. im. I.G.Petrovskogo, v.4, 1978,
      p.83-136 , English transl. Selecta Math. Sov., v.5, 1986, 141-199.


\bibitem[5]{IYa07}Yu. Ilyashenko,  S. Yakovenko, Lectures on
    analytic     differential equations, GTM, AMS, 2007, 639 pp


\bibitem[6]{LN} A. Lins Neto, Fibers of he Baum-Bott map for
    foliations of degree two on $\mathbb P^2$


\bibitem[7]{LNSS99}  A.Lins Neto, P.Sad, B.Scardua, On
    topological rigidity of projective foliations,     Bull.
    Soc. Math. France , 126, no. 3 (1998),
    381--406


\bibitem[8]{M06} V. Moldavskis, New generic properties of real and
    complex dynamical systems, PhD Thesis,  Cornell University, 2007


\bibitem[9]{N94}  I.Nakai, Separatrices for nonsolvable
    dynamics on ${\cc},0)$ , Ann. Inst. Fourier,  44 no. 2 (1994), 569--599

 \bibitem[10]{Na82}  V. Naishul', Topological equivalence of differential
 equations in $C^2$ and $CP^2$ , Vestnik Moskov. Univ. Ser 1 Mat. Mekh.
 (1981) no 4, 8--11

\bibitem[11]{O96} L. Ortiz Bobadilla, Topological equivalence of
    linear autonomous equations in $\cc^m$ with Jordan blocks,
    Trans. Moscow Math. Soc, 57 (1996), 67--91


\bibitem[12]{P06} Quadratic Vector Fields in CP2 with Solvable
    Monodromy Group at Infinity, Tr. Mat. Inst. im. V.A. Steklova,
    Ross. Akad. Nauk 254, (2006) 130–161  [Proc. Steklov Inst. Math.
    254, 121–151 (2006)].


\bibitem[13]{S84} A. Shcherbakov, Topological and Analytic
    Conjugacy of Noncommutative Groups of Germs of Conformal
    Mappings, Tr. Semin. im. I.G. Petrovskogo, 10,  (1984), 170–196
    [J. Sov. Math. 35, 2827–2850 (1986)].



\bibitem[14]{S06}   A. Shcherbakov,
    Dynamics of Local Groups of Conformal Mappings and Generic
    Properties of Differential Equations on $\cc^2$, Tr. Mat. Inst. im.
    V.A. Steklova, Ross. Akad. Nauk, v. 254, (2006), 111–129  [Proc.
    Steklov Inst. Math. 254, 103–120 (2006)].


\end{thebibliography}
\end{document}